\documentclass[11pt]{article}
\usepackage{mathrsfs}
\usepackage{}
\usepackage{amsfonts}
\usepackage{amssymb}
\usepackage{graphicx}
\usepackage{amsmath}\usepackage[T1]{fontenc}\usepackage{lmodern}
\usepackage{esint}
\usepackage{enumitem}

\def\sqr#1#2{{\vcenter{\vbox{\hrule height.#2pt
              \hbox{\vrule width.#2pt height#1pt \kern#1pt \vrule
width.#2pt}
              \hrule height.#2pt}}}}
\def\signed #1{{\unskip\nobreak\hfil\penalty50
              \hskip2em\hbox{}\nobreak\hfil#1
              \parfillskip=0pt \finalhyphendemerits=0 \par}}
\def\endpf{\signed {$\sqr69$}}
\def\3n{\negthinspace \negthinspace \negthinspace }
\def\2n{\negthinspace \negthinspace }
\def\1n{\negthinspace }

\def\ms{\medskip}

\def\q{\quad}

\def\({\Big (}
\def\){\Big )}
\def\[{\Big[}
\def\]{\Big]}

\def\be{\begin{equation}}
\def\bel{\begin{equation}\label}
\def\ee{\end{equation}}
\def\bea{\begin{eqnarray}}
\def\eea{\end{eqnarray}}
\def\bt{\begin{theorem}}
\def\et{\end{theorem}}
\def\bc{\begin{corollary}}
\def\ec{\end{corollary}}
\def\bl{\begin{lemma}}
\def\el{\end{lemma}}
\def\bp{\begin{proposition}}
\def\ep{\end{proposition}}
\def\br{\begin{remark}}
\def\er{\end{remark}}
\def\ba{\begin{array}}
\def\ea{\end{array}}
\def\bd{\begin{definition}}
\def\ed{\end{definition}}

\newtheorem{lemma}{Lemma}[section]
\newtheorem{remark}{Remark}[section]

\newtheorem{theorem}{Theorem}[section]
\newtheorem{corollary}{Corollary}[section]

\newtheorem{definition}{Definition}[section]
\newtheorem{proposition}{Proposition}[section]

\makeatletter
   
   \@addtoreset{equation}{section}
\makeatother
\allowdisplaybreaks
\begin{document}

\title{\bf Controllability of quasilinear parabolic equations under multiplicative mobile controls}

\author{Lingyang Liu\thanks{School of Mathematical Sciences, South China Normal University, Guangzhou, 510631, China. E-mail address: liuly@scnu.edu.cn \ms } \q
 \q }

\date{}

\maketitle

\begin{abstract}
This paper addresses the controllability of a class of quasi-linear parabolic equations governed by multiplicative controls with mobile support. To prove the existence of such a control forcing the solution to rest at time $T>0$, we first establish the decay property of solutions for the uncontrolled system. Unlike the case of the linear heat equation, the nonlinearity in the principal part of the operator introduces significant challenges. These difficulties necessitate a novel approach, ultimately leading us to solve the controllability problem within the framework of classical solutions. Through a carefully constructed smooth transition, we demonstrate that there exists a multiplicative control driving the state exactly to rest at time $t=T$.
\end{abstract}

\noindent{\bf Keywords: Nonlinear parabolic equation, Controllability, Multiplicative control,  Nonnegative additive control }

\section{Introduction and main result}\label{sec1}

It is known that linear, semilinear and quasilinear parabolic systems--in the absence of constraints--are controllable in any positive time (see, e.g., \cite{Bec,LZ}). However,
controllability under nonnegativity constraints in parabolic systems does not hold when the
time horizon is too short. The first relevant work appears to be \cite{LTZ}. There, the authors studied the existence of a positive minimal controllability time for the heat equation under unilateral state or control constraints. In a later study, the authors of \cite{PZ} extended the analysis of \cite{LTZ} to semilinear equations with $C^1$ nonlinearities. In \cite{NC}, the same question for quasilinear parabolic equations was investigated. It is also worth mentioning that the nonnegative controllability of a class of Newtonian filtration equations with interior degeneracy was discussed in \cite{LiG}. It was shown that the target set is approximately controllable when the control acts on the entire spatial domain, while controllability fails if the control is restricted to a strict open subset. These works focus on the constrained controllability of parabolic equations governed by additive controls acting on a proper subset of the boundary or the domain.

Regarding multiplicative controllability in parabolic systems, the authors of \cite{FK} proved the approximate controllability of the one-dimensional heat equation governed by either multiplicative or nonnegative additive controls with mobile local support. Interestingly, they showed that for any $  T > 0 $, the heat equation lacks controllability under static support. Nevertheless, global approximate controllability was achieved with mobile support. In \cite{LG}, the authors generalized the results of \cite{FK} to the case of heat equations with boundary degeneracy.

Let $ T > 0  $ and denote $   Q_T:= (0,1)\times (0, T)  $. In this paper, we consider the following Dirichlet problems:
\begin{equation}\label{1.1}
\left\{\begin{array}{ll}
y_{t}-(a(y)y_{x})_x=u(x,t)y(x,t)\chi_{\omega(t)}(x) &\text{in}\ Q_T,\\[2mm]
y(0,t)=y(1,t)=0 &\text{in}\ (0,T),\\[2mm]
y(x,0)=y_0(x)&\text{in}\ (0,1),
\end{array}\right.
\end{equation}
where $y_0\in L^2_+(0,1)=\{\phi\in L^2(0,1):\phi(x)\geq0\ \text{a.e.}\ \text{in}\ (0,1)\}$,  $  \omega(t)   $ is a non-empty open subset of $(0,1)$ for every $   t \in (0, T)  $, $\chi_{\omega(t)}$
denotes the characteristic function of $\omega(t)$, $u$ is a multiplicative control that, at each time $t$, is supported in ${\omega(t)}$, and $a$ is assumed to be a real-valued function satisfying
\begin{equation}\label{c1}
a\in C^2(\mathbb R), \quad  \inf\limits_{r\in\mathbb R}a(r)>0\quad\text{and}\quad \sup\limits_{r\in\mathbb R}\left\{|a'(r)|, |a''(r)|\right\}<\infty.
\end{equation}

We address the null controllability of quasilinear parabolic equations under multiplicative control. Our objective is to find an $L^\infty$-control steering the state of the parabolic system to zero exactly at any finite time $T>0$. It is observed that the parabolic system \eqref{1.1} is not controllable in $L^2_+(0,1)$ with a static support. Hence, we allow the control support to move in order to achieve the null controllability of \eqref{1.1}. Specifically, the authors of \cite{FK} applied the method of separation of variables to design controls that achieve the desired steering, even allowing for the control to be zero at the end of the process. However, the method of constructing an explicit control function seems inapplicable to quasilinear parabolic problems, due to the presence of time-space dependent coefficients in the principal part.

In the sequel, we assume that
\begin{equation}\label{r}
\omega(t)=(r(t),r(t)+l),
\end{equation}
where the length $l \in (0, 1)$ is fixed and $r :[0,T ] \rightarrow [0,1-l]$ is a function that defines the position of the support
at time $t$. By definition, the support $  \omega(t) $ is static if $ r $ is constant and mobile if $  r  $ is time-varying.
Furthermore, we denote by $\tilde{y}_{y_0} $ the solution of the uncontrolled system
and by $  PC[0, T]$ the set of piecewise constant functions $  r: [0, T] \to [0, 1-l]  $ with at most a finite number of discontinuities.

First, we will show the lack of controllability for (\ref{1.1}) with a static subdomain.
\begin{theorem}\label{t1.1}
Let us assume that $y_0\in L^2_+(0,1)$, $y_0\not\equiv 0$, and $\omega$ is
a fixed proper subinterval of $(0, 1)$. Denote by $y_{y_0,u,\omega}$ the solution of system (\ref{1.1}). Then, for any $T > 0$, the set
\begin{equation*}
F_{y_0,\omega}(T) :=\left\{y_{y_0,u,\omega}(x,T) : x \in (0, 1), u \in L^\infty(Q_T), \omega(t) \equiv \omega, \forall t \in (0, T )\right\}
\end{equation*}
is not dense in $L^2_+(0,1)$.
\end{theorem}

The main result is stated as follows.
\begin{theorem}\label{1.4-}
System (\ref{1.1}) is locally approximately null controllable in $L^2_+(0,1)$. That is, for any $  T > 0  $, $ \epsilon> 0  $, $  \sigma\in(0, 1)$, $y_d\in L^2_+(0,1)$, there exists a positive constant $ \gamma $, such that for any given initial datum $ y_0\in L^2_+(0,1)\cap C^{2+\frac{\sigma}{2}}([0,1])$, satisfying $|y_0|_{C^{2+\frac{\sigma}{2}}([0,1])} \leq \gamma$ and the first order compatibility condition, there exists a bilinear control $ u \in C^{\frac{\sigma}{2}, \frac{\sigma}{4}}(\overline{Q_T})   $ such that the corresponding solution  of (\ref{1.1}) satisfies
\begin{equation*}
\| u(T)-y_d \|_{L^2(0,1)} < \epsilon.
\end{equation*}
\end{theorem}

The rest of this paper is organized as follows. In Section \ref{sec3.1}, we prove the lack of controllability for the system with a static control support. Section \ref{sec3.2} is devoted to proving the constrained approximate controllability of the linearized system via an additive, nonnegative, locally distributed control. In Section \ref{sec3.3}, we analyze the multiplicative approximate controllability of the linearized system. Finally, Section \ref{sec3} contains the proof of our main result, namely, Theorem \ref{1.4-}.

\section{Noncontrollability of system \eqref{1.1} in $L^2_+(0,1)$}\label{sec3.1}
In this section, we show that system \eqref{1.1} is not approximately controllable in $L^2_+(0,1)$ at any time $T > 0$ when the control acts on a static subinterval of $(0,1)$.
To prove this result, we first consider quasilinear parabolic equations of the form:
\begin{equation}\label{e2.0}
\left\{\begin{array}{ll}
y_{t}-(a(y)y_{x})_x=u(x,t)y(x,t)+v(x,t) &\text{in}\ Q_T,\\[2mm]
y(x,0)=y_0(x)&\text{in}\ (0,1),
\end{array}\right.
\end{equation}
where $a$ satisfies the condition \eqref{c1}.

With respect to \eqref{e2.0}, we present the following lemma:
\begin{lemma}[Maximum Principle]\label{l1.2}
Let us assume that $y_0 \in L^2(0, 1)$, $v \in L^2(Q_T)$ and $u \in L^\infty(Q_T)$.
Let $y \in L^2(0, T; H^1(0, 1)) \cap C([0, T]; L^2(0, 1))$ be a solution of equation \eqref{e2.0}. Then the following properties hold:

\begin{enumerate}
 \item[(i)]  If $y_0 \ge 0$ in $(0, 1)$, $v \geq 0$ in $Q_T$, $y(0,t)\ge 0$ and $y(1,t)\ge 0$ in $(0,T)$, then $y(x, t) \ge 0$ a.e. in $Q_T$;

 \item[(ii)] if $v \equiv 0$, $u \leq 0$ in $Q_T$, $y(0, \cdot)$, $y(1, \cdot)\in L^\infty(0, T )$ and $y_0 \in L^\infty(0, 1)$, then $\sup\limits_{{Q_T}} |y(x, t)|\leq\sup\limits_{{\Gamma_T}} |y(x, t)|$, where $\Gamma_T = ([0,1] \times \{0\}) \cup (\{0,1\} \times (0,T]) $;

 \item[(iii)]  If $y \in L^2(0, T; H_0^1(0, 1))$ with $y_0 \ge 0$ in $(0, 1)$ ($y_0 \not\equiv 0$), $v \geq 0$ in $Q_T$ and $u = u(x) \le 0$ in $(0, 1)$, then $y(\cdot, t) > 0$ in $(0, 1)$ for all $t \in (0, T]$.
\end{enumerate}
\end{lemma}

By the fixed point technique and the maximum principle for linear parabolic equations (see \cite[Ch. III, \S7, pp. 181-191]{LSU} and \cite[pp. 21-23]{POV}), one can prove the results of Lemma \ref{l1.2}. Regarding the classical setting, see also \cite[Ch. I, \S2, pp. 11-25]{LSU} and \cite[pp. 121-122]{Gu}.

Moreover, the following comparison principle is valid:
\begin{lemma}\label{l1.1}
   Let $  y_1 $ and $  y_2 $ be two solutions of equation \eqref{e2.0} such that
    $  y_1(0, t) \leq y_2(0, t)   $ and $   y_1(1, t) \leq y_2(1, t)   $ in $  (0, T)   $.
        Then $  y_1 \leq y_2  $ in $  Q_T   $.
\end{lemma}

Now, we are in a position to give a proof of Theorem \ref{t1.1}.

\emph{Proof of Theorem \ref{t1.1}.}  We will show that the set $F_{y_0,\omega}(T)$ (see \eqref{1.4-}) is not dense in $L^2_+(0,1)$. To prove this, we argue by contradiction. Let $y_d\in L^2_+(0,1)$. Suppose that there exists a sequence
$\{u_k\}_{k=1}^\infty \subset L^\infty(Q_T)$ such that:
\begin{equation*}
y_{y_0,u_k,\omega}(\cdot,T)\rightarrow y_d\quad \text{in}\ L^2_+(0,1)  \ \text{as}  \  k\rightarrow+\infty.
\end{equation*}
Put $y_d=0$. Then for any $\widetilde{\omega} = (a, b) \subset (0,1)\setminus\omega$, we have
\begin{equation*}
y_{y_0,u_k,\omega}(\cdot,T)\mid_{\widetilde{\omega}}\rightarrow 0\quad \text{in}\ L^2_+(\widetilde{\omega})  \ \text{as}  \  k\rightarrow+\infty.
\end{equation*}
It follows from Lemma \ref{l1.2} (i) that $y_{y_0,u_k,\omega}\ge 0$ in $Q_T$. Let $y_{y_0,u_k,\omega}\mid_{\widetilde{\omega}}$ be the solution to the following parabolic problem:
\begin{equation*}
\left\{\begin{array}{ll}
y_{t}-(a(y)y_{x})_x=0,  &\text{in}\ \widetilde{\omega}\times(0,T),\\[2mm]
y(a,t)=h_0(t), \ y(b,t)=h_1(t), &\text{in}\ (0,T),\\[2mm]
y(x,0)=y_0(x),&\text{in}\ \widetilde{\omega},
\end{array}\right.
\end{equation*}
where $h_0(t)$ and $h_1(t)$ are nonnegative boundary values.

Denote by $\hat{y}$ the solution to the homogeneous parabolic problem:
\begin{equation*}
\left\{\begin{array}{ll}
\hat{y}_{t}-(a(\hat{y})\hat{y}_{x})_x=0,  &\text{in}\ \widetilde{\omega}\times(0,T),\\[2mm]
\hat{y}(a,t)=0, \ \hat{y}(b,t)=0, &\text{in}\ (0,T),\\[2mm]
\hat{y}(x,0)=y_0(x),&\text{in}\ \widetilde{\omega}.
\end{array}\right.
\end{equation*}
If $y_0 \geq 0$ and $y_0 \not\equiv 0$ in $\widetilde{\omega}$, then Lemma \ref{l1.2} (iii) implies that $\hat{y}(\cdot,T)>0$ in $\widetilde{\omega}$.

Applying Lemma \ref{l1.1} in $\widetilde{\omega}\times(0,T)$, we see that for any $k\in \mathbb N_+$,
\begin{equation*}
y_{y_0,u_k,\omega}(\cdot,T)\mid_{\widetilde{\omega}}\geq\hat{y}(\cdot,T)>0\quad \text{in}\ \widetilde{\omega}.
\end{equation*}
Thus, we arrive at the contradiction
\begin{equation*}
0=\lim\limits_{k\rightarrow+\infty}\|y_{y_0,u_k,\omega}(\cdot,T)\|_{L^2(\widetilde{\omega})}\geq\|\hat{y}(\cdot,T)\|_{L^2(\widetilde{\omega})}>0,
\end{equation*}
which proves the original claim. \endpf

\section{Approximately controllability of linear parabolic equations by additive control}\label{sec3.2}
This section addresses the additive nonnegative controllability problem with mobile support:
\begin{equation}\label{2.0+}
\left\{\begin{array}{ll}
y_{t}-(by_{x})_x = v(x, t)\chi_{\omega(t)}(x), & \text{in } Q_T, \\[2mm]
y(0, t) = y(1, t) = 0, & \text{in } (0, T), \\[2mm]
y(x, 0) = y_0(x), & \text{in } (0, 1),
 \end{array}\right.
\end{equation}
where  $b\in C^{1,1}(\overline{Q_T})$ satisfies for some constant $\rho> 0$,
\begin{equation}\label{b}
b(x, t)\geq \rho \quad\forall(x, t)\in \overline{Q_T},
\end{equation}
and $\chi_{\omega(t)}$ denotes the characteristic function with moving support given by  \eqref{r}.

We establish the following approximate controllability result for \eqref{2.0+}:
\begin{theorem}\label{l1.3}
Let us assume that ${y}_0\in L^2_+(0,1)$, $y_0  \not\equiv 0$.
Then, for any $T > 0$, the set
\begin{equation*}
F^+_{1,0,mb}(T) :=\left\{y_{y_0,v,r}(\cdot, T) : v \in L^2_+(Q_T)\ \text{and} \ r\in PC[0, T]\right\}
\end{equation*}
is dense in $L^2_+(0,1)$.
\end{theorem}

To prove Theorem \ref{l1.3}, we first establish the following lemma.
\begin{lemma}\label{t2.0}
For $y_0 \in H_0^1(0,1) \cap L^2_+(0,1)$, denote by $\tilde{y}$ the unique solution of \eqref{2.0+} with $v= 0$. Then the following estimate holds:
\begin{equation}\label{3.3}
\| \tilde{y}(\cdot,T)-y_0 \|_{L^{2}(0, 1)} \leq \sqrt{|b|_{C(\overline{Q_T})}T}\| y_{0,x} \|_{L^{2}(0, 1)}.
\end{equation}
\end{lemma}
\emph{Proof of Lemma \ref{t2.0}.} Define $Y(x, t)=\tilde{y}(x, t)-y_0(x) $. Then $Y $ satisfies the initial-boundary value problem
\begin{equation}\label{28}
\left\{\begin{array}{ll}
Y_{t}-(bY_{x})_x = (by_{0,x})_x, & \text{in } Q_T, \\[2mm]
Y(0, t) = Y(1, t) = 0, & \text{in } (0, T), \\[2mm]
Y(x, 0) = 0, & \text{in } (0, 1).
 \end{array}\right.
\end{equation}
Multiplying both sides of the first equation of \eqref{28} by $Y$ and integrating over $(0, 1)$,  we obtain
\begin{equation*}
\frac{d}{2dt}\int^1_0 Y^2dx+\int^1_0bY^2_{x} dx= -\int^1_0by_{0,x}Y_xdx.
\end{equation*}
Combining (\ref{b}) with the Cauchy inequality, we find
\begin{equation*} \label{6}
\frac{d}{dt}\int^1_0 Y^2dx\leq \int^1_0by^2_{0,x}dx.
\end{equation*}
Since $ Y(x, 0) = 0$, it follows that
\begin{equation*}
 \| Y(\cdot,T) \|_{L^{2}(0, 1)} \leq \sqrt{|b|_{C(\overline{Q_T})}T}\| y_{0,x} \|_{L^{2}(0, 1)},
\end{equation*}
which completes the proof of Lemma \ref{t2.0}. \endpf

\begin{remark}
Estimate \eqref{3.3} implies that $\tilde{y}(\cdot, T)$ is close to $y_0$ in the $L^{2}$ norm provided the time $T$ is sufficiently small.
\end{remark}

Now, consider the following controlled linear parabolic equation
\begin{equation}\label{2.2}
\left\{\begin{array}{ll}
y_t - (by_x)_x = v(x, t) & \text{in }\omega\times(0,T), \\[2mm]
y = 0 & \text{on } \partial\omega\times(0,T), \\[2mm]
y(0) = y_0 & \text{in } \omega,
\end{array}\right.
\end{equation}
where $  \omega  $ is a fixed proper subinterval of $ (0,1)   $.

We have the following regional nonnegative controllability result for \eqref{2.2}.
\begin{lemma}\label{t1.3}
Let $y_0\in L^2(\omega)$ and denote by $y(\cdot;v)$ the solution of \eqref{2.2}. Then, for any $T > 0$, the set
$
F :=\left\{y(T;v) :  v \in L^2_+(\omega\times(0,T))\right\}
$
is dense in $y(T;0)+L^2_+(\omega)$.
\end{lemma}
\emph{Proof of Lemma \ref{t1.3}.} Without loss of generality, we assume $  y_0 \equiv 0  $. Arguing by contradiction, suppose that there exists $  y_d \in L^2_+(\omega)  $ such that $ y_d \notin \overline{F} $. Note that $  \overline{F}  $ is a closed and convex set. By the Hahn-Banach theorem (in its geometrical form), we can separate $  y_d $ from $  \overline{F} $, i.e., there exist $   \alpha \in \mathbb{R}  $ and $  g \in L^{2}(\omega) $ such that
\begin{equation}\label{2.1}
\int_{\omega} y(T; v) g \, dx \leq 0 < \alpha < \int_{\omega} y_d g \, dx \quad \forall v \in L^2_+(\omega\times(0,T)).
\end{equation}
Let $  q \in C([0, T]; L^{2}(\omega))   $ be the solution of the auxiliary backward problem:
\begin{equation}\label{e2.2}
\left\{\begin{array}{ll}
-q_t - (bq_x)_x= 0 & \text{in } \omega\times(0,T), \\[2mm]
q = 0 & \text{on } \partial\omega\times(0,T), \\[2mm]
q(T) = g & \text{in } \omega.
\end{array}\right.
\end{equation}
Multiplying both sides of the first equation of (\ref{e2.2}) by $  y   $ and integrating by parts, we obtain
\begin{equation*}
0 \ge \int_{\omega} g y(T; v) \, dx= \int_{\omega\times(0,T)} q v \, dx dt \quad \forall v \in L^2_+(\omega\times(0,T)).
\end{equation*}
This implies that $  q \leq 0  $ a.e. in $\omega\times(0,T)$. In particular, $  g = q(T) \leq 0   $. This contradicts (\ref{2.1}), because (\ref{2.1}) implies that $  \int_{\omega} y_d g \, dx > 0  $ while $   y_d \geq 0  $ and $  g \leq 0  $ would force this integral to be non-positive. \endpf

\begin{remark}
By the classical convolution technique, we can deduce that there exists a sequence $\{v_k\}_{k=1}^\infty \subset C_+(\overline{\omega\times(0,T)})$ such that $v_k\rightarrow v$ in $L^2(\omega\times(0,T))$ as $k \rightarrow +\infty$. Let $y$ and $y_k$ denote the unique solutions of \eqref{2.2} corresponding to the controls $v$ and $v_k$, respectively. Subtracting the two equations, multiplying the resulting equation by the difference $y - {y_{k}}$, and integrating by parts yields
\begin{equation*}
\begin{array}{ll}
\displaystyle\quad\frac{1}{2} \int_\omega (y(x,T) - y_{k}(x,T))^2 dx + \int_{\omega\times(0,T)}  b(x,t) \left( (y(x,t) - y_k(x,t))_x \right)^2 dxdt  \\[3.5mm]
=\displaystyle \int_{\omega\times(0,T)}  (v(x,t) - v_k(x,t))(y(x,t) - y_k(x,t)) dxdt,
\end{array}
\end{equation*}
Using condition  (\ref{b}) and applying the Cauchy and Poincar\'{e} inequalities, we obtain
\begin{equation*}
\|y(\cdot,T) - y_{k}(\cdot,T) \|_{L^2(\omega)} \le C(\rho,meas(\omega) ) \| v_k - v \|_{L^2(\omega\times(0,T))}.
\end{equation*}
Consequently, the $L^2$-norm of the difference between the solutions at time $T$ can be made arbitrarily small by taking the $L^2$-norm of the difference of the controls sufficiently small.
\end{remark}

Now, we are in a position to give a proof of Theorem \ref{l1.3}.

\emph{Proof of Theorem \ref{l1.3}.}
First, let us fix $M$ as the smallest natural number satisfying $M \cdot l \geq 1$, where $l$ is the (fixed) length of the moving subinterval $\omega(t)$ (see (\ref{r})).
Given any element $y_d \in L_+^2(0, 1)$, we can decompose it into $M$ ``pieces'' localized in disjoint subintervals of length at most $l$, as follows:
\begin{equation*}
y_{d_1}(x) = \begin{cases} y_d(x), & x \in (0, l), \\ 0, & x \in [l, 1), \end{cases}
\end{equation*}
\begin{equation*}
y_{d_j}(x) = \begin{cases} y_d(x), & x \in ((j-1)l, jl), \\ 0, & \text{otherwise}, \end{cases}
\end{equation*}
for $ j = 2, \ldots, M-1 $,
\begin{equation*}
y_{d_M}(x) = \begin{cases} y_d(x), & x \in ((M-1)l, 1), \\ 0, & \text{otherwise}. \end{cases}
\end{equation*}
Clearly, each $ y_{d_j}\in L^2_+(0,1)$ and
\begin{equation*}
y_d (x) = \sum_{j=1}^M y_{d_j}(x), \quad \text{a.e. } x \in (0, 1).
\end{equation*}
Moreover, given $ \epsilon > 0 $, there exist $\hat{y}_{d_j} \in H_0^1(0,1) \cap L^2_+(0,1)$ for $ j = 1, \ldots, M  $, with its support $ \omega_j $ contained in the corresponding one of each $ y_{d_j} $ and verifying
\begin{equation*}
\sum_{j=1}^M \| \hat{y}_{d_j} - y_{d_j} \|_{L^2(0,1)} \leq \frac{\epsilon}{2}.
\end{equation*}

We now present a sequential argument, using the controls provided by Lemma \ref{t1.3}, as follows: given $  \hat{y}_{d_j} $, $  \delta_j \in (0, \delta_{j-1}/2)  $, $  j = 1, \ldots, M   $, there exist controls
\begin{equation*}
\hat{v}_j(x, t) = \begin{cases}
0, & t \in (0, T - 2\delta_j), \\
{v}_j(x, t), & t \in (T - 2\delta_j, T - \delta_j), \\
0, & t \in (T - \delta_j, T),
\end{cases}
\end{equation*}
for $  j = 1, \dots, M-1   $, and
\begin{equation*}
\hat{v}_M(x, t) = \begin{cases}
0, & t \in (0, T - \delta_M), \\
{v}_M(x, t), & t \in (T - \delta_M, T),
\end{cases}
\end{equation*}
for which
\begin{equation*}
\| \tilde{y}_{0,\hat{v}_j,\omega_j}(\cdot, T) - \hat{y}_{d_j} \|_{L^2(0,1)} \leq \frac{\epsilon}{2M}, \quad j = 1, \dots, M.
\end{equation*}
We finish the proof by taking
\begin{equation*}
v(x, t) = \sum_{j=1}^M \hat{v}_j(x, t), \qquad \tilde{y}(x, t) = \sum_{j=1}^M \tilde{y}_{0,\hat{v}_j,\omega_j}(x, t)
\end{equation*}
and noticing that by the linearity
\begin{equation*}
\begin{aligned}
\| \tilde{y}(\cdot, T) - y_d \|_{L^2(0,1)} &\leq \left\| \tilde{y}(\cdot, T) - \sum_{j=1}^M \hat{y}_{d_j} \right\|_{L^2(0,1)} + \left\| \sum_{j=1}^M \hat{y}_{d_j} - y_d \right\|_{L^2(0,1)} \\
&\leq \sum_{j=1}^M \| \tilde{y}_{0,\hat{v}_j,\omega_j}(\cdot, T) - \hat{y}_{d_j} \|_{L^2(0,1)} + \sum_{j=1}^M \| \hat{y}_{d_j} - y_{d_j} \|_{L^2(0,1)} \leq \epsilon.
\end{aligned}
\end{equation*}
Note that in each time subinterval $   (T - \delta_j, T - \delta_{j+1})   $, the control $v(x, t)  $ is equal to $  \hat{v}_{j+1}(x, t)   $ and hence its support is included in $   \omega(t) = (r(t), r(t) + l)   $, with $   r(t) \in PC[0, T]  $ given by
\begin{equation*}
r(t) = \begin{cases}
0, & t \in (0, T - \delta_1), \\
l, & t \in (T - \delta_1, T - \delta_2), \\
2l, & t \in (T - \delta_2, T - \delta_3), \\
\qquad\quad\vdots & \\
jl, & t \in (T - \delta_j, T - \delta_{j+1}), \\
\qquad\quad\vdots & \\
(M-2)l, & t \in (T - \delta_{M-2}, T - \delta_{M-1}), \\
1-l, & t \in (T - \delta_{M-1}, T).
\end{cases}
\end{equation*}
This implies that $v(x, t) = v(x, t)\chi_{\omega(t)}(x)   $ and $  \tilde{y} = \tilde{y}_{0,v,r}  $, as desired. \endpf

\section{Controllability of linear parabolic equations governed by multiplicative control}\label{sec3.3}

In this section, we establish the controllability of the following linear parabolic equation:
\begin{equation}\label{2.0}
\left\{\begin{array}{ll}
y_{t}-(by_{x})_x = u(x,t)y(x,t)\chi_{\omega(t)}(x), & \text{in } Q_T, \\[2mm]
y(0, t) = y(1, t) = 0, & \text{in } (0, T), \\[2mm]
y(x, 0) = y_0(x), & \text{in } (0, 1),
 \end{array}\right.
\end{equation}
where  $b\in C^{1,1}(\overline{Q_T})$ satisfies the uniform ellipticity condition
\begin{equation}\label{1.3}
b(x, t)\geq \rho > 0 \qquad \forall (x, t)\in \overline{Q_T},
\end{equation}
and $\chi_{\omega(t)}$ denotes the characteristic function of the time-dependent set $\omega(t)$ defined in \eqref{r}.

Combining Lemma \ref{l1.2}-(i), (iii) with Theorem \ref{l1.3}, and following the argument of Theorem 4.1 in \cite{FK}, we can prove the following approximate nonnegative controllability result for \eqref{2.0}.
\begin{lemma}\label{t2.2}
Let us assume that $y_0\in L^2_+(0,1)$, $y_0\not\equiv0$. Then, for any $T > 0$, the set
\begin{equation*}\label{1.4}
F^+_{2,y_0,mb}(T) =\left\{y_{y_0,u,r}(x,T) : x \in (0, 1), u \in L^\infty_+(Q_T)\ \text{and} \ r\in PC[0, T]\right\}
\end{equation*}
is dense in $\tilde{y}_{y_0}   (   \cdot, T )+L^2_+(0,1)$.
\end{lemma}

We now prove the main global controllability result.
\begin{theorem}\label{t2.1}
Let us assume that $y_0\in L^2_+(0,1)$, $y_0 \not\equiv 0$. Then, for any $T > 0$, the set
\begin{equation*}
F_{2,y_0,mb}(T) =\left\{y_{y_0,u,r}(\cdot, T)  : u \in L^\infty(Q_T),\ r \in PC[0, T] \right\}
\end{equation*}
is dense in $L^2_+(0,1)$.
\end{theorem}
\emph{Proof of Theorem \ref{t2.1}.}
Our objective is to approximate any target function $y_d \in L^2_+(0,1)$, even when $y_d \not\in\tilde{y}_{y_0}(\cdot, T) + L^2_+(0,1)$.
To this end, we fix $M$ (as defined previously) as the smallest natural number such that $M \cdot l \geq 1$, where $l$ is the length of the moving subinterval $\omega(t)$.
With this $M$, we consider intermediate times $0 \leq T_1 \leq T_2 \leq \cdots \leq T_M < T$. Let $y_1(x, t)$ denote the unique solution of the following system
\begin{equation*}
\left\{\begin{array}{ll}
y_{t}-(by_{x})_x=-m_1y(x,t)\chi_{(0,l)}(x),  &\text{in}\ (0,1)\times(0,T_1),\\[2mm]
y(0,t)=0, \ y(1,t)=0, &\text{in}\ (0,T_1),\\[2mm]
y(x,0)=y_{0}(x),&\text{in}\ (0,1),
\end{array}\right.
\end{equation*}
where $m_1 \geq 0$ will be chosen appropriately later. Assuming now that $y_{j-1}$ is already known, $y_j$ will denote the solution of the system
\begin{equation*}
\left\{\begin{array}{ll}
y_{t}-(by_{x})_x=-m_jy(x,t)\chi_{((j-1)l,jl)}(x),  &\text{in}\ (0,1)\times(T_{j-1},T_j),\\[2mm]
y(0,t)=0, \ y(1,t)=0, &\text{in}\ (T_{j-1},T_j),\\[2mm]
y(x,T_{j-1})=y_{j-1}(x,T_{j-1}),&\text{in}\ (0,1),
\end{array}\right.
\end{equation*}
for $j = 2, \dots, M-1$, with constants $m_j \geq 0$ also to be chosen. Finally, $y_M$ will denote the solution of the problem
\begin{equation*}
\left\{\begin{array}{ll}
y_{t}-(by_{x})_x=-m_My(x,t)\chi_{((M-1)l,1)}(x),  &\text{in}\ (0,1)\times(T_{M-1},T_M),\\[2mm]
y(0,t)=0, \ y(1,t)=0, &\text{in}\ (T_{M-1},T_M),\\[2mm]
y(x,T_{M-1})=y_{M-1}(x,T_{M-1}),&\text{in}\ (0,1),
\end{array}\right.
\end{equation*}
with $m_M \geq 0$ to be chosen. The following result will be crucial in this section:
\begin{proposition}\label{l2.3}
With the previous notations, given $\epsilon > 0$, there exist nonnegative constants $m_1, \dots, m_M$ and intermediate times $0 \leq T_1 \leq T_2 \leq \cdots \leq T_M < T$ such that
\begin{equation*}
\| y_M(\cdot, T_M) \|_{L^2(0,1)} \leq \frac{\epsilon}{2}.
\end{equation*}
\end{proposition}

We defer the proof of Proposition \ref{l2.3} until later in this section. Now, we proceed with the proof of Theorem \ref{t2.1}.  Let $\hat{y}_1(x, t)$ denote the unique solution of the system
\begin{equation*}
\left\{\begin{array}{ll}
y_{t}-(by_{x})_x=0,  &\text{in}\ (0,1)\times(T_{M},T),\\[2mm]
y(0,t)=0, \ y(1,t)=0, &\text{in}\ (T_{M},T),\\[2mm]
y(x,T_{M})=y_{M}(x,T_{M}),&\text{in}\ (0,1),
\end{array}\right.
\end{equation*}
It is known from \cite{Liu} that
\begin{equation}\label{M(t)}
\| \hat{y}_1(\cdot, T) \|_{L^2(0,1)} \leq \| y_M(\cdot, T_M) \|_{L^2(0,1)} \leq \frac{\epsilon}{2},
\end{equation}
where the last inequality comes from Proposition \ref{l2.3}.

By Lemma \ref{l1.2}-(i), we have that $y_j(x, t) \geq 0$ a.e. $(x, t) \in Q_T$ for $j = 1, \dots, M$. In particular, $y_M(x, T_M) \geq 0$ a.e. $x \in (0,1)$.
Assume that $y_d \in L^2_+(0,1)$ and $\epsilon> 0$ are fixed. Applying now Lemma \ref{t2.2} to system
(\ref{2.0}) in the domain $(0, 1)\times(T_M, T )$ with initial datum $y_M(\cdot, T_M )$, we deduce the existence of $u_2 \in L^\infty_+ ((0, 1)\times(T_M, T))$ and $r_2 \in PC[T_M, T]$ such that the solution $\hat{y}_2(x, t)$ of the system
\begin{equation*}
\left\{\begin{array}{ll}
{y}_{t}-(b{y}_{x})_x=u_2(x,t)y(x,t)\chi_{\omega(t)}(x),  &\text{in}\ (0, 1)\times(T_M,T),\\[2mm]
{y}(0,t)=0, \ {y}(1,t)=0, &\text{in}\ (T_M,T),\\[2mm]
{y}(x,T_M)=y_M(x,T_M),&\text{in}\ (0, 1),
\end{array}\right.
\end{equation*}
satisfies
\begin{equation}\label{3.5-}
\|\hat{y}_2(\cdot,T)-(y_d+\hat{y}_1(\cdot,T))\|_{L^{2}(0, 1)}<\frac{\epsilon}{2}.
\end{equation}
We finish the proof by defining
\begin{equation*}
y_{y_0,u,r}(x,t)=\left\{\begin{array}{ll}
\hat{y}_1(x,t),  &\text{in}\ (0, 1)\times(0,T_M],\\[2mm]
\hat{y}_2(x,t),  &\text{in}\ (0, 1)\times[T_M,T),
\end{array}\right.
\end{equation*}
with
\begin{equation*}
u(x,t)=\left\{\begin{array}{ll}
-m_1,  &\text{in}\ (0, 1)\times(0,T_1),\\[2mm]
-m_2,  &\text{in}\ (0, 1)\times(T_{1},T_2),\\[2mm]
\vdots & \\[2mm]
-m_j,  &\text{in}\ (0, 1)\times(T_{j-1},T_j),\\[2mm]
\vdots & \\[2mm]
-m_M,  &\text{in}\ (0, 1)\times(T_{M-1},T_M),\\[2mm]
u_2(x,t),  &\text{in}\ (0, 1)\times(T_M,T),
\end{array}\right.
\end{equation*}
and
\begin{equation*}
r(t)=\left\{\begin{array}{ll}
0,  &\text{in}\ (0,T_1),\\[2mm]
l,  &\text{in}\ (T_{1},T_2),\\[2mm]
\vdots & \\[2mm]
jl,  &\text{in}\ (T_{j},T_{j+1}),\\[2mm]
\vdots & \\[2mm]
(M-2)l,  &\text{in}\ (T_{M-2},T_{M-1}),\\[2mm]
1-l,  &\text{in}\ (T_{M-1},T_M),\\[2mm]
r_2(t),  &\text{in}\ (T_M,T).
\end{array}\right.
\end{equation*}
Clearly, $u\in L^\infty(Q_T )$ and $r\in PC[0, T]$. Moreover, thanks to (\ref{M(t)}) and (\ref{3.5-}), it is satisfied
\begin{equation*}
\begin{array}{ll}
\|y_{y_0,u,r}(\cdot,T)-y_d\|_{L^{2}(0,1)}\!\!\!&=\|y_{y_0,u,r}(\cdot,T)-(y_d+\hat{y}_1(\cdot,T))+\hat{y}_1(\cdot,T)\|_{L^{2}(0,1)}\\[2mm]
&\leq\|y_{y_0,u,r}(\cdot,T)-(y_d+\hat{y}_1(\cdot,T))\|_{L^{2}(0,1)}+\|\hat{y}_1(\cdot,T)\|_{L^{2}(0,1)}\\[2mm]
&=\|\hat{y}_2(\cdot,T)-(y_d+\hat{y}_1(\cdot,T))\|_{L^{2}(0,1)}+\|\hat{y}_1(\cdot,T)\|_{L^{2}(0,1)}<\epsilon,
\end{array}
\end{equation*}
as we were looking for. \endpf

As a preliminary step for the proof of Proposition \ref{l2.3}, we first establish the following result.
\begin{lemma}\label{p1.2}
Let $\sigma\in(0, 1)$. Assume that $u \in C^1(\overline{Q_T})$ with $u(x,t) \leq0$ for all $(x,t) \in \overline{Q_T}$ and that $y_0 \in C^{2+\sigma}[0, 1]$  satisfies $y_0(x) \geq 0$ for all $x \in[0, 1]$ and the first order compatibility condition. Consider the problem
\begin{equation}\label{1.2}
\left\{\begin{array}{ll}
y_{t}-(by_{x})_x = u(x,t)y(x, t), & \text{in } Q_T, \\[2mm]
y(0, t) = y(1, t) = 0, & \text{in } (0, T), \\[2mm]
y(x, 0) = y_0(x), & \text{in } (0, 1),
 \end{array}\right.
\end{equation}
where $b\in C^{1,1}(\overline{Q_T})$ satisfies
\eqref{1.3}. Then, the classical solution to \eqref{1.2} possesses the following properties:
\begin{enumerate}
\item[(a)] $\| y_t(\cdot, t) \|_{L^2(0, 1)}  \leq \left[2\left(\| (b(\cdot, 0)y'_{0})' \|^2_{L^2(0, 1)}+K_2 \right)+K_4K_3Te^{K_1T}\right]^\frac{1}{2}\ \forall t \in [0, T]$, where $ K_1$, $K_2$, $K_3$ and $K_4$ are positive constants depending only on $ \|y_0\|_{L^2(0, 1)}  $,  $\|\sqrt{b(\cdot, 0)}y'_0 \|_{L^2(0, 1)}$, $|b_t|_{C(\overline{Q_T})}$, $|u|_{C(\overline{Q_T})}$, $ |u_t|_{C(\overline{Q_T})}$ and $\rho$;
\item[(b)] $0 \leq y_x(0, t) \leq e^{\Big(1+\frac{|b_x|_{ C(\overline{Q_T})}}{\rho}\Big)} \cdot \max\limits_{x \in [0, 1]} \left\{ y_0'(x) e^{y_0(x)} \right\} \quad \forall t \in [0, T]$.
\end{enumerate}
\end{lemma}
\emph{Proof of Lemma \ref{p1.2}.}
In the case $y_0 \equiv 0$, the result is clearly valid because $y \equiv 0$. So, let us suppose that $y_0 \not\equiv 0$.
The existence and uniqueness of a classical solution $y \in C^{2+\sigma, 1+\frac{\sigma}{2}}(\overline{Q_T})$ for problem (\ref{1.2}) is a consequence of \cite[Theorem 5.2, p. 320]{LSU}. To prove (a), we first multiply both sides of the first equation of (\ref{1.2}) by $  y $, integrate over $  (0, 1)  $,  and utilize the assumption $  u \leq 0    $  to obtain
\begin{equation*}
\frac{1}{2} \frac{d}{dt}\int_{0}^{1} y^2  dx+\int_{0}^{1} by_x^2  dx\leq 0.
\end{equation*}
In view of \eqref{1.3}, it follows that
\begin{equation}\label{1.5}
\| y(\cdot, t) \|_{L^2(0, 1)} \leq \| y_0 \|_{L^2(0, 1)}\quad \forall t \in [0, T].
\end{equation}
Now, multiplying equation (\ref{1.2}) by $  y_t  $ and integrating over $  (0, 1)  $, we deduce
\begin{equation*}
 \int_{0}^{1} y_t^2  dx + \frac{1}{2} \frac{d}{dt}\int_{0}^{1} by_x^2  dx  =  \frac{1}{2}\int_{0}^{1} b_ty_x^2   dx+\int_{0}^{1} uyy_t dx.
\end{equation*}
Applying Cauchy's inequality and noting that $    b_t  $  and $   u  $ are bounded, we arrive at
\begin{equation*}
 \frac{d}{dt}\int_{0}^{1} by_x^2  dx \leq \|b_t(\cdot, t)\|_{L^{\infty}(0,1)} \int_{0}^{1} y_x^2   dx+\|u(\cdot, t)\|^2_{L^{\infty}(0,1)} \int_{0}^{1} y^2 dx.
\end{equation*}
This, together with inequalities (\ref{1.3}) and (\ref{1.5}), yields
\begin{equation*}
 \frac{d}{dt}\int_{0}^{1} by_x^2  dx \leq  K_1\int_{0}^{1} by_x^2   dx+K_2,
\end{equation*}
where
\begin{equation*}
K_1= \frac{|b_t|_{C(\overline{Q_T})}}{\rho}  \quad\  \text{and}  \quad\  K_2 = |u|^2_{C(\overline{Q_T})} \| y_0 \|^2_{L^2(0, 1)}.
\end{equation*}
By Gronwall's inequality, we obtain
\begin{equation}\label{1.6}
\| y_x(\cdot, t) \|_{L^2(0, 1)}  \leq  \sqrt{K_3}e^{\frac{K_1t}{2}}\quad \forall t \in [0, T],
\end{equation}
where
$$K_3= \frac{K_2+\|\sqrt{b(\cdot, 0)}y'_0 \|^2_{L^2(0, 1)} }{\rho}.$$

Next, we differentiate (\ref{1.2}) with respect to the time variable to deduce that $y_t$ is the solution of the
problem
\begin{equation}\label{2.8}
\left\{\begin{array}{ll}
z_t-(bz_{x})_x-uz=(b_ty_{x})_x+u_ty &\text{in}\ Q_T,\\[2mm]
z(0, t) = z(1, t) = 0 & \text{in } (0, T), \\[2mm]
 z(x,0) =  z_0(x) &\text{in}\ (0,1),
\end{array}\right.
\end{equation}
where $ z_0(x)=\left( b(x,0)  y'_{0}(x) \right)' + u(x,0)y_0(x)$.
Multiplying both sides of the first equation of (\ref{2.8}) by $z$ and integrating over $(0,1)$, we have
\begin{equation*}
\frac{1}{2} \frac{d}{dt} \int_{0}^{1} z^2  dx + \int_{0}^{1} bz_x^2  dx- \int_{0}^{1} uz^2 dx = -  \int_{0}^{1} b_ty_{x}z_x  dx+\int_{0}^{1} u_tyz  dx.
\end{equation*}
Since $u$ is nonpositive and $b$ satisfies (\ref{1.3}),  it follows that
\begin{equation}\label{4.16}
\frac{1}{2} \frac{d}{dt} \int_{0}^{1} z^2  dx + \rho \int_{0}^{1} z_x^2  dx\leq  \int_{0}^{1} |b_ty_{x}z_x|  dx+\int_{0}^{1} |u_tyz|  dx.
\end{equation}
We estimate the right-hand side terms of \eqref{4.16}. First, an application of Cauchy's inequality with parameter $\varepsilon=\rho$  yields
\begin{equation*}
\int_{0}^{1} |b_ty_{x}z_x|  dx\leq   \frac{\rho}{2} \int_{0}^{1} z_x^2  dx+\frac{\|b_t(\cdot,t)\|^2_{L^\infty(0,1)}}{2\rho}\int_{0}^{1} y_{x}^2  dx.
\end{equation*}
For the second term, using Cauchy's inequality and Poincar\'{e}'s inequality, we obtain
\begin{equation*}
\int_{0}^{1} |u_tyz|  dx\leq \frac{\rho}{2} \int_{0}^{1} z_x^2   dx+\frac{\|u_t(\cdot,t)\|^2_{L^\infty(0,1)}}{2\rho}\int_{0}^{1} y_{x}^2   dx.
\end{equation*}
Substituting these two inequalities into \eqref{4.16} leads to
\begin{equation}\label{1.7}
 \frac{d}{dt}\int_{0}^{1} z^2  dx \leq  K_4\int_{0}^{1} y^2_{x}  dx,
\end{equation}
where
\begin{equation*}
K_4= \frac{|b_t|^2_{C(\overline{Q_T})}+|u_t|^2_{C(\overline{Q_T})}}{\rho}.
\end{equation*}
Combining (\ref{1.6}) with (\ref{1.7}), we get
\begin{equation*}
\| y_t(\cdot, t) \|_{L^2(0, 1)}  \leq \left[2\left(\| (b(\cdot, 0)y'_{0})' \|^2_{L^2(0, 1)}+K_2 \right)+K_4K_3Te^{K_1T}\right]^\frac{1}{2} \quad \forall t \in [0, T].
\end{equation*}

The lower bound in (b) follows directly from the definition, the homogeneous boundary condition and the nonnegativity of $y$. To derive the upper bound, we use Bernstein's method (see \cite{LSU}, pp. 414 and 537), by introducing the auxiliary function
\begin{equation*}
w(x, t) = e^{y(x,t)} + M e^{1-\Big(1+\frac{|b_x|_{ C(\overline{Q_T})}}{\rho}\Big)x}-1,
\end{equation*}
where $M = \max\limits_{x \in [0,1]} \big\{ y'_0(x) e^{y_0(x)} \big\}\frac{\rho}{\rho+{|b_x|_{ C(\overline{Q_T})}}} e^{\frac{|b_x|_{ C(\overline{Q_T})}}{\rho}}$. Note that $M > 0$ due to the hypotheses on $y_0$.

By direct calculation, it can be checked that $  w \in C^{2,1}(\overline{Q_T})  $ and satisfies the following PDE:
\begin{equation*}
\begin{array}{ll}
\displaystyle &w_t(x, t) - (bw_{x})_{x}(x, t) \\[3mm]= \!\!\!\!&\displaystyle\left[ u(x, t)y(x, t) - b(y_x(x, t))^2 \right] e^{y(x,t)} \\[3mm]
\displaystyle&-  M \Big(1+\frac{|b_x|_{ C(\overline{Q_T})}}{\rho}\Big)e^{1-\Big(1+\frac{|b_x|_{ C(\overline{Q_T})}}{\rho}\Big)x}\left[ \Big(1+\frac{|b_x|_{ C(\overline{Q_T})}}{\rho}\Big)b-b_x \right].
\end{array}
\end{equation*}
Let $ (x_0, t_0)  $ be a point in $  Q_T  $ where $   w   $ attains its maximum. We will show that $  x_0 = 0   $. If $  (x_0, t_0) \in Q_T $, then at this point we have $  w_t(x_0, t_0) = 0  $, $  w_x(x_0, t_0) = 0  $ and $  w_{xx}(x_0, t_0) \leq 0   $. Substituting into the above PDE leads to the contradiction:
\begin{equation*}
0 \leq w_t(x_0, t_0) - (bw_{x})_{x}(x_0, t_0) \leq -  M \Big(1+\frac{|b_x|_{ C(\overline{Q_T})}}{\rho}\Big)e^{1-\Big(1+\frac{|b_x|_{ C(\overline{Q_T})}}{\rho}\Big)x_0}b(x_0, t_0)<0.
\end{equation*}
The same reasoning can be used to exclude the case $ t_0 = T $, since the only difference is that $ w_t(x_0, t_0) \geq 0 $. Consequently, $ (x_0, t_0) $ must be a point on the parabolic boundary of $ Q_T $, that is, $ x_0 = 0 $, $ x_0 = 1 $, or $ t_0 = 0 $.

The possibilities $ x_0 = 1 $ and $ t_0 = 0 $ can be excluded as follows. First, we observe that $ w(0, t)  \geq w(1, t) $. Second, at the initial time $ t = 0 $, we have $ w(x, 0) = e^{y_0(x)} +  M e^{1-\Big(1+\frac{|b_x|_{ C(\overline{Q_T})}}{\rho}\Big)x}-1 $. Given the choice of $ M $, the spatial derivative at $ t=0 $ is
\begin{equation*}
 w_x(x, 0) = y'_0(x) e^{y_0(x)} - M\Big(1+\frac{|b_x|_{ C(\overline{Q_T})}}{\rho}\Big) e^{1-\Big(1+\frac{|b_x|_{ C(\overline{Q_T})}}{\rho}\Big)x}\leq  0  \quad\ \forall x \in [0, 1].
\end{equation*}
Therefore, $ w(x, 0)  $ is non-increasing in $  [0, 1]  $, so $  w(x, 0) \leq w(0, 0) = M e   $ for all $  x \in [0, 1]  $. This implies that the maximum of $   w  $ over $  \overline{Q_T } $ is attained at $   x=0   $, where $  w(0, t) = M e   $ for all $  t \in [0, T]   $. Hence, $  w_x(0, t) \leq 0   $, or equivalently, $ y_x(0, t) \leq M e\Big(1+\frac{|b_x|_{ C(\overline{Q_T})}}{\rho}\Big)  $, for all $  t \in [0, T]  $. \endpf

Following an argument similar to that of Corollary 4.4 in \cite{FK}, we have
\begin{corollary}\label{cor1.2}
Under the same hypotheses of Lemma \ref{p1.2}, except that $u \in W^{1, \infty}(0, T; L^\infty(0,1))$, the conclusions $(a)-(b)$
remain valid, for almost every point.
\end{corollary}

Now, we are in a position to give a proof of Proposition \ref{l2.3}.

\emph{Proof of Proposition \ref{l2.3}.}
The conclusion is clearly valid if the $L^2$-norm of the initial datum is small enough. Our aim is to address the general case.
Without loss of generality, we assume that the initial data appearing in the proof are nonnegative regular functions with compact support in $(0, 1)$.
We first prove the existence of nonnegative constants $m_1,\ldots, m_M$, intermediate times $0 \leq T_1 \leq T_2 \leq \cdots \leq T_M < T $, and corresponding solutions $ y_j$ for $j = 1, \ldots, M$, such that the conclusion of Proposition \ref{l2.3} holds. To this end, we consider the family of problems
\begin{equation}\label{3.8}
\left\{\begin{array}{ll}
y_{t}-(by_{x})_x=-my(x,t)\chi_{(0,l)}(x),  &\text{in}\ Q_T,\\[2mm]
y(0,t)=0, \ y(1,t)=0, &\text{in}\ (0,T),\\[2mm]
y(x,0)=y_0(x),&\text{in}\ (0,1),
\end{array}\right.
\end{equation}
and denote its solution by $y_{m,1}$ for each $m \geq 0$. Multiplying the PDE in (\ref{3.8}) by $y_{m,1}$ and integrating by parts over $Q_T$, we obtain
\begin{equation*}
\begin{array}{ll}
\displaystyle\frac{1}{2} \int_0^1 (y_{m,1}(x, T))^2  dx - \frac{1}{2} \int_0^1 (y_0(x))^2  dx =\!\!\!\!&\displaystyle -\iint_{Q_T}b(x, t)\left( (y_{m,1})_x(x, t) \right)^2   dx dt\\[4mm]
 &\displaystyle- m \int_0^T \int_0^l(y_{m,1}(x, t))^2 dx dt.
\end{array}
\end{equation*}
Since the term $-\iint_{Q_T}b(x, t)\left( (y_{m,1})_x(x, t) \right)^2   dx dt$ is nonpositive, it follows that
\begin{equation*}
\int_0^T \int_0^l (y_{m,1}(x, t))^2 dx dt\leq \frac{1}{2m} \int_0^1 (y_0(x))^2  dx.
\end{equation*}
From this inequality, we deduce that the sequence $g_m(t) = \sqrt{\int_0^l (y_{m,1}(x, t))^2 dx} $ converges to $0$ in $L^2(0, T)$ as $m \to +\infty $. Consequently, by taking a subsequence if necessary, we have $g_m(t) \to 0$ for almost every $t \in (0, T)$ as $m \to +\infty $.
Now, given $\epsilon > 0$, there exist $T_1 \in (0, T)$ and a sufficiently large $m_1 > 0$ such that $g_{m_1}(T_1) < \frac{\epsilon}{2\sqrt{2M - 1}}$. Denoting $y_1 = y_{m_1,1}$, this implies the following estimate:
\begin{equation}\label{2.3+}
\int_0^{l} (y_1(x, T_1))^2 dx \leq \frac{\epsilon^2}{4(2M - 1)}.
\end{equation}

Following the same scheme, we now consider the family of problems:
\begin{equation}\label{e2.1}
\left\{\begin{array}{ll}
y_{t}-(by_{x})_x=-my(x,t)\chi_{(l,2l)}(x),  &\text{in}\ (0,1)\times(T_1,T),\\[2mm]
y(0,t)=0, \ y(1,t)=0, &\text{in}\ (T_1,T),\\[2mm]
y(x,T_1)=y_{01}(x),&\text{in}\ (0,1),
\end{array}\right.
\end{equation}
where $y_{01}(x) = y_1(x, T_1)$, and $y_1$ is the function obtained in the previous step. Denoting its solution by $y_{m,2}$ for each $m \geq 0$, and repeating the same argument, we can find $T_2 \in (T_1, T)$ and a sufficiently large $m_2 > 0$ such that the function $y_2 = y_{m_2, 2}$ satisfies
\begin{equation}\label{2.3}
\int_l^{2l} (y_2(x, T_2))^2 dx \leq \frac{\epsilon^2}{4(2M - 1)}.
\end{equation}
In fact, the time $T_2$ can be chosen arbitrarily close to $T_1$. More precisely, we can select $T_2$ such that
\begin{equation}\label{2.2-}
0 < T_2 - T_1 \leq \min\left\{ \frac{\epsilon^2}{8(2M - 1)C_1},\ T - T_1 \right\},
\end{equation}
where
\begin{equation}\label{2.}
\begin{array}{ll}
\displaystyle C_1 = \| y_{01} \|_{L^{\infty}(0,1)}\Bigg(\!\!\!\!\!&\displaystyle|b(0, \cdot)|_{ C([0, T])}e^{\Big(1+\frac{|b_x|_{ C(\overline{Q_T})}}{\rho}\Big)}  \max\limits_{x \in [0, 1]} \left\{ y_{01}'(x) e^{y_{01}(x)} \right\} \\[5mm]
&+ \left[2\left(\| (b(\cdot, T_1)y'_{01})' \|^2_{L^2(0, 1)}+K_2 \right)+K_4K_3Te^{K_1T}\right]^\frac{1}{2}\Bigg),
\end{array}
\end{equation}
with the constants $  K_1  $, $  K_2  $, $K_3$ and $K_4$ are defined as in Lemma \ref{p1.2}. Note that the terms appearing in (\ref{2.}) are directly related to the conclusions $(a)$ and $(b)$ of that lemma.

Under these conditions, we aim to prove that
\begin{equation*}
\int_0^{2l} (y_2(x, T_2))^2 dx \leq \frac{3\epsilon^2}{4(2M - 1)}.
\end{equation*}
In view of inequality (\ref{2.3}), it suffices to prove that
\begin{equation}\label{2.5}
\int_0^{l} (y_2(x, T_2))^2 dx \leq \frac{\epsilon^2}{2(2M - 1)}.
\end{equation}
We multiply the PDE in (\ref{e2.1}) (with $ m = m_2$) by $y_2$ and integrate by parts over the domain $(0, l) \times (T_1, T_2)$. In this process, we use the boundary condition at $x = 0$ and note that the intervals $(0, l)$ and $(l, 2l)$ are disjoint (so the source term vanishes on $ (0, l) $), thus obtaining
\begin{equation*}
\begin{array}{ll}
\displaystyle\frac{1}{2} \int_0^{l} (y_2(x, T_2))^2 dx - \frac{1}{2} \int_0^{l} (y_{01}(x))^2 dx =\!\!\!\!&\displaystyle -\int_{T_1}^{T_2} \int_0^{l} b(x, t)\left( (y_2)_x(x, t) \right)^2 dx dt \\[4mm]
 &\displaystyle+ \int_{T_1}^{T_2} b(l, t)(y_2)_x(l, t) y_2(l, t) dt.
\end{array}
\end{equation*}
Combining this identity with the previous estimate (\ref{2.3+}), we derive the inequality:
\begin{equation*}
\begin{array}{ll}
\displaystyle\int_0^{l} (y_2(x, T_2))^2 dx \!\!\!&\displaystyle\leq \int_0^{l} (y_{01}(x))^2 dx + 2 \int_{T_1}^{T_2}  b(l, t)(y_2)_x(l, t) y_2(l, t) dt\\[5mm]
&\displaystyle\leq \frac{\epsilon^2}{4(2M - 1)}+ 2 \int_{T_1}^{T_2}  b(l, t)(y_2)_x(l, t) y_2(l, t) dt.
\end{array}
\end{equation*}
Therefore, to verify (\ref{2.5}), it is sufficient to prove that the boundary term satisfies:
\begin{equation}\label{2.7}
2 \int_{T_1}^{T_2} b(l, t)(y_2)_x(l, t) y_2(l, t) dt \leq \frac{\epsilon^2}{4(2M - 1)}.
\end{equation}
We now verify this bound by exploiting the specific choice of $ T_2 $ (see (\ref{2.2-})--(\ref{2.})). Using H\"{o}lder's inequality and Corollary \ref{p1.2} (with $u(x) = -m_2 \cdot \chi_{(l,2l)}(x)$, which is a bounded function) for problem (\ref{e2.1}), we find
\begin{equation*}
\begin{array}{ll}
\displaystyle b(l, t)(y_2)_x(l, t) \!\!\!&\displaystyle=  b(0, t)(y_2)_x(0, t) + \int_0^l (b(y_2)_{x})_x(x, t)dx\\[5mm]
&\displaystyle=  b(0, t)(y_2)_x(0, t)+ \int_0^l (y_2)_t(x, t)dx\\[5mm]
&\displaystyle\leq  b(0, t)(y_2)_x(0, t) +l^\frac{1}{2} \| (y_2)_t(\cdot, t) \|_{L^2(0, l)}   \\[5mm]
&\displaystyle\leq |b(0, \cdot)|_{ C([0, T])}e^{\Big(1+\frac{|b_x|_{ C(\overline{Q_T})}}{\rho}\Big)}  \max\limits_{x \in [0, 1]} \left\{ y_{01}'(x) e^{y_{01}(x)} \right\} \\[5mm]
&\quad+ \left[2\left(\| (b(\cdot, T_1)y'_{01})' \|^2_{L^2(0, 1)}+K_2 \right)+K_4K_3Te^{K_1T}\right]^\frac{1}{2}
\overset{\text{def}}{=} C_2.
\end{array}
\end{equation*}
We can now combine this estimate with the fact that $0 \leq y_2 \leq \|y_{01}\|_{L^\infty(0,1)}$ (thanks to the maximum principle), to get
\begin{equation*}
2\int_{T_1}^{T_2} b(l, t)(y_2)_x(l, t)y_2(l, t)dt \leq 2\int_{T_1}^{T_2} C_2 \cdot \|y_{01}\|_{L^\infty(0,1)}dt = 2(T_2 - T_1)C_1 \leq \frac{\epsilon^2}{4(2M - 1)},
\end{equation*}
which is exactly (\ref{2.7}).

Finally, we can repeat the argument to select the constants $m_j$, the times $T_j \in [T_{j-1}, T)$, and the solutions $y_j$, for $j = 1, 2, \dots, M-1$, verifying
\begin{equation*}
\int_0^{jl} (y_j(x, T_j))^2dx \leq \frac{(2j - 1)\epsilon^2}{4(2M - 1)},
\end{equation*}
and arrive at
\begin{equation*}
\int_0^1 (y_M(x, T_M))^2dx \leq \frac{\epsilon^2}{4},
\end{equation*}
as desired.   \endpf

\section{Local controllability of system \eqref{1.1}}\label{sec3}
In this section, we prove the approximate controllability of the quasi-linear parabolic equation (\ref{1.1}) using a fixed-point argument.
For this, let $R>0$ be given and set
\begin{equation*}\label{K}
 \mathcal D :=  \left\{ z \in C^{1+\frac{\sigma}{2},1+\frac{\sigma}{4}}(\overline{Q_T}) : |z|_{C^{1+\frac{\sigma}{2},1+\frac{\sigma}{4}}(\overline{Q_T})} \leq R,\ z(0) = y_0 \right\}.
\end{equation*}
Clearly, $ \mathcal D $ is a nonempty, convex, and compact subset of $ L^2(Q_T) $ for small initial data $ y_0 $.

We consider the mapping $\Lambda$: $ \mathcal D \rightarrow 2^\mathcal D $, defined as
follows: for any $ z \in \mathcal D $,
\begin{equation*}
\begin{array}{ll}
 \Lambda(z) =  y,
\end{array}
\end{equation*}
where $y$ is the solution to the following linearized equation of (\ref{1.1}):
\begin{equation}\label{homo}
\left\{\begin{array}{ll}
y_{t}-(a(z)y_{x})_x=u(x,t)y(x,t) &\text{in}\ Q_T,\\[2mm]
y(0,t)=y(1,t)=0 &\text{in}\ (0,T),\\[2mm]
y(x,0)=y_0(x)&\text{in}\ (0,1).
\end{array}\right.
\end{equation}
By Theorem \ref{t2.1}, for any $ z \in \mathcal D $, there exists a control $u\in L^{\infty}({Q_T})$ (in fact, by smooth approximation, even
$u\in C^{\frac{\sigma}{2},\frac{\sigma}{4}}(\overline{Q_T})$) such that the corresponding solution $y$ of \eqref{homo} satisfies
\begin{equation}\label{y}
 \|y(T)-y_d\|_{L^{2}(0,1)} \leq\varepsilon.
\end{equation}

Now, we are in a position to give a proof of Theorem \ref{1.4-}.

\emph{Proof of Theorem \ref{1.4-}.} We analyze the controllability of \eqref{1.1} in
the frame of classical solutions.
Specifically, we prove that $\Lambda$ has a fixed point by Kakutani's fixed point theorem, which yields the approximate controllability of system (\ref{1.1}).

\emph{Step 1.} We show that $ \Lambda(\mathcal D) \subseteq \mathcal D$.
By the Schauder theory of second-order linear parabolic equations  (see \cite[Theorem 7.2.24]{WYW}), we have the estimate
\begin{equation}\label{3.10}
|y|_{C^{2+\frac{\sigma}{2},1+\frac{\sigma}{4}}(\overline{Q_{T}})}\leq C|y_0|_{C^{2+\frac{\sigma}{2}}([0,1])},
\end{equation}
where $C$ depends on $|a|_{C([-R,R])}$, $|u|_{C^{\frac{\sigma}{2},\frac{\sigma}{4}}(\overline{Q_T})}$ and $T$. Moreover, by checking the proof of  Proposition \ref{l2.3}, it follows that
\begin{equation}\label{3.21}
|u|_{C^{\frac{\sigma}{2},\frac{\sigma}{4}}(\overline{Q_T})}\leq C(\varepsilon,\|y_0\|_{L^2(0,L_0)}).
\end{equation}
Consequently, there exists a sufficiently small constant $\gamma > 0$ such that $|y|_{C^{2+\frac{\sigma}{2},1+\frac{\sigma}{4}}(\overline{Q_{T}})}\leq R$ provided that $ |y_0|_{C^{2+\frac{\sigma}{2}}([0,1])} \leq \gamma$; that is, $y \in \mathcal D$.

\emph{Step 2.} We prove that the mapping $\Lambda$ is upper semicontinuous.
Assume that $z_n\in \mathcal D$, $z_n\rightarrow z$, $y^{(n)}\in \Lambda(z_n)$, and $y^{(n)}\rightarrow y$ as $n\rightarrow\infty$. We aim to show that $y \in  \Lambda(z)$.

For any $z_n\in  D$, let us consider the following system
\begin{equation}\label{3.15}
\left\{\begin{array}{ll}
y^{(n)}_{t}-(a(z_n)y^{(n)}_{x})_x=u_n(x,t)y^{(n)}(x,t), & (x,t)\in Q_{T},\\[2mm]
y^{(n)}(0,t)=0, \ y^{(n)}({1},t)=0, &t\in(0,T),\\[2mm]
y^{(n)}(x,0)=y_0(x),&x\in(0,1).
\end{array}\right.
\end{equation}
From (\ref{3.10}), we see that $\{y^{(n)}\}_{n=1}^{\infty}$ is uniformly bounded in ${C^{2+\frac{\sigma}{2},1+\frac{\sigma}{4}}(\overline{Q_{T}})}$. Moreover, $\{u_n\}_{n=1}^{\infty}$ satisfies (\ref{3.21}). It follows from Arzel\`{a}-Ascoli theorem that there exists a function $u\in C^{\frac{\sigma}{2},\frac{\sigma}{4}}(\overline{Q_T})$ such that $u_n\rightarrow u$ {in} $C(\overline{Q_T})$. Furthermore,
\begin{equation*}
\begin{array}{ll}
z_n\rightarrow  z\quad  \text{in} \quad  C^{1,1}(\overline{Q_T}),\\[3mm]
y^{(n)}\rightarrow y\quad  \text{in} \quad C^{2,1}(\overline{Q_{T}}).
\end{array}
\end{equation*}
Letting $n\rightarrow\infty$ in \eqref{3.15},
we find that the following equation holds:
\begin{equation}\label{3.22}
\left\{\begin{array}{ll}
y_{t}-(a(z)y_{x})_x=u(x,t)y(x,t), & (x,t)\in Q_{T},\\[2mm]
y(0,t)=0, \ y({1},t)=0, &t\in(0,T),\\[2mm]
y(x,0)=y_0(x),&x\in(0,1).
\end{array}\right.
\end{equation}
Since equation \eqref{3.22} admits at most one classical solution, we conclude that $y \in  \Lambda(z)$.

\emph{Step 3.} For any $z \in \mathcal D $, $\Lambda(z)$ is a nonempty, convex and compact subset of $ L^2(Q_T) $.

Therefore, by Kakutani's fixed point theorem, there exists a $y \in  \mathcal D$ such that $y \in \Lambda(y)$. This means that for system (\ref{1.1}) we can find a control $  u \in C^{\frac{\sigma}{2},\frac{\sigma}{4}}(\overline{Q_T})$ with $\operatorname{supp} u(x,t) \subseteq \omega(t)\times [0,T]$ such that the corresponding solution $ y $ satisfies (\ref{y}). Moreover, the cost of the control $u $ fulfills estimate (\ref{3.21}).

Thus the proof of Theorem \ref{1.4-} is complete. \endpf


\begin{thebibliography}{1}{\small}


\bibitem{Bec} M. Beceanu, {\it Local exact controllability of the diffusion equation in one dimension},  Abstr. Appl. Anal. 14 (2003) 793-811.

%\bibitem{CH} J. R. Cannon, C. D. Hill,  {\it Existence, uniqueness, stability, and monotone dependence in a Stefan problem for the heat equation}, J. Math. Mech. 17 (1967) 1-19.


\bibitem{FK} L. A. Fern\'{a}ndez, A. Y. Khapalov, {\it Controllability properties for the one-dimensional Heat equation under multiplicative or nonnegative additive controls with local mobile support}, ESAIM: Control Optim. Calc. Var. 18 (2012) 1207-1224.

%\bibitem{FHL} E. Fern\'{a}ndez-Cara, F. Hern\'{a}ndez, J. L\'{\i}maco, {\it Local null controllability of a 1D Stefan problem}, Bull. Braz. Math. Soc.
%(N. S.) 50 (3) (2019) 745-769.

%\bibitem{FLTD} E. Fern\'{a}ndez-Cara, J. Limaco, Y. Thamsten, D. Menezes, {\it Local null controllability of a quasi-linear system and related numerical experiments}, ESAIM Control Optim. Calc. Var. 29 (2023) 27.

%\bibitem{Fri} A. Friedman, {\it Partial Differential Equations}, Holt, Rinehart and Winston, New York, 1969.


%\bibitem{FI} A. V. Fursikov, O. Yu. Imanvilov,  {\it Controllability of Evolution Equations}, Lecture Notes Ser. 34, Seoul National University, Seoul, Korea, 1996.


\bibitem{Gu} L. Gu, {\it Parabolic Partial Differential Equation of Second Order}, Xiamen University Press,
Fujian Province, China, 1995 (in Chinese).

\bibitem{LSU} O. A. Lady\v{z}enskaja, V. A. Solonnikov, N. N. Ural'ceva,  {\it Linear and Quasilinear
Equations of Parabolic Type}, Transl. Math. Monogr. 23, AMS, Providence, RI, 1968.

\bibitem{LG} L. Li, H. Gao, {\it Approximate controllability for degenerate heat equation with bilinear control}, J. Syst. Sci. Complex. 34 (2) (2021) 537-551.

\bibitem{Liu} L. Liu, {\it Stability of a class of parabolic equations in non-cylindrical domains}, Appl. Anal. to appear.

\bibitem{LiG} X. Liu, H. Gao, {\it Controllability of a class of Newtonian filtration equations with control and state constraints}, SIAM J. Control Optim. 46 (6) (2007), 2256-2279.


\bibitem{LZ} X. Liu, X. Zhang, {\it
Local controllability of multidimensional quasi-linear parabolic equations}, SIAM J. Control Optim. 50 (2012) 2046-2064.

\bibitem{LTZ} J. Loh\'{e}ac, E. Tr\'{e}lat, E. Zuazua, {\it Minimal controllability time for the heat equation under unilateral state or control constraints}, Math. Models Methods Appl. Sci. 27 (9) (2017) 1587-1644.


\bibitem{NC} M. R. Nu\~{n}ez-Ch\'{a}vez, {\it Controllability under positive constraints for quasilinear parabolic PDEs}, Math. Control Relat. Fields 12 (2) (2022) 327-341.


\bibitem{PZ} D. Pighin, E. Zuazua, {\it Controllability under positivity constraints of semilinear heat equations}, Math. Control Relat. Fields 8 (3-4) (2018) 935-964.


\bibitem{POV} A. I. Prilepko, D. G. Orlovsky, I. A. Vasin, {\it Methods for solving inverse problems in mathematical physics}, Marcel Dekker
Inc., New York, 2000.


%\bibitem{WLL} L. Wang, Y. Lan, P. Lei, {\it
%Local null controllability of a free-boundary problem for the quasi-linear 1D parabolic equation},  J. Math. Anal. Appl. 506 (2) (2022) 1-26.


%\bibitem{WLW} L. Wang, P. Lei, Q. Wu, {\it
%Null controllability of a 1D Stefan problem for the heat equation governed by a multiplicative control}, Systems Control Lett. 171 (2023) 1-8.

\bibitem{WYW} Z. Wu, Y. Yin, C. Wang, {\it Elliptic and Parabolic Equations}, World Scientific Publishing Co. Pvt. Ltd., Hackensack, NJ, 2006.

%\bibitem{LL} L. Liu,  X. Liu {\it Controllability and observability of some coupled stochastic parabolic systems}, World Scientific Publishing Co. Pvt. Ltd., Hackensack, NJ, 2006.



\end{thebibliography}
\end{document}